\documentclass{article}
\usepackage{amssymb,latexsym,amsmath,epsfig,amsthm,hyperref,tipa,comment}

\newtheorem{standardfareylemma}{Lemma}
\newtheorem{newfareylemma}[standardfareylemma]{Lemma}
\newtheorem{upperbound}{Theorem}
\newtheorem{lowerbound}[upperbound]{Theorem}
\newtheorem{generalthm}[upperbound]{Theorem}
\newtheorem*{maxkconjecture}{Conjecture}
\newtheorem*{stupidlystrongconjecture}{Conjecture}
\newtheorem{notbothsmall}[standardfareylemma]{Lemma}
\newtheorem{summatorytotient}[standardfareylemma]{Lemma}
\newtheorem{dress}[standardfareylemma]{Lemma}
\newtheorem{sumofsmallreciprocals}[standardfareylemma]{Lemma}

\begin{document}

\vspace*{-2cm}

\Large
 \begin{center}
Improved bounds for the Mayer-Erd\H os phenomenon on similarly ordered Farey fractions \\ 

\hspace{10pt}

\large
Wouter van Doorn \\

\hspace{10pt}

\end{center}

\hspace{10pt}

\normalsize

\vspace{-25pt}

\centerline{\bf Abstract}
Let $\frac{a_1}{b_1}, \frac{a_2}{b_2}, \ldots$ be the Farey fractions of order $n$. We then prove that the inequality $(a_l - a_k)(b_l - b_k) \ge 0$ holds for all $k$ and $l > k$ with $l-k \le \left(\frac{1}{12} - o(1) \right)n$, sharpening an old result by Erd\H os. On the other hand, we will show that for all $n \ge 4$ there are $k, l$ with $k < l < k + \frac{n}{4} + 5$ for which the product $(a_l - a_k)(b_l - b_k)$ is negative.

\section{Introduction}
If two fractions $\frac{a}{b}$ and $\frac{a'}{b'}$ are such that the product $(a' - a)(b' - b)$ is non-negative, then we say that $\frac{a}{b}$ and $\frac{a'}{b'}$ are similarly ordered. For example, $\frac{2}{5}$ and $\frac{3}{7}$ are similarly ordered, while $\frac{2}{5}$ and $\frac{3}{4}$ are not. With this definition in mind, let $\frac{a_1}{b_1}, \frac{a_2}{b_2}, \ldots$ be the Farey sequence of order $n \ge 4$ and let $f(n)$ be the largest integer such that $\frac{a_k}{b_k}$ and $\frac{a_l}{b_l}$ are similarly ordered for all $k$ and $l$ with $|l - k| \le f(n)$. The condition $n \ge 4$ here ensures that the Farey sequence of order $n$ actually contains fractions (e.g. $\frac{1}{4}$ and $\frac{2}{3}$) which are not similarly ordered, so that $f(n)$ is unambiguously defined. \\

In \mbox{\cite{may1}} Mayer proved $f(n) \ge 3$ for all $n \ge 5$, which he subsequently improved in \mbox{\cite{may2}} to $f(n) \rightarrow \infty$ if $n \rightarrow \infty$. This was further improved by Erd\H os in \mbox{\cite{erd}}, where he showed $f(n) > cn$ for some suitable constant $c$. Moreover, his proof showed that one can take $c = \frac{1}{400}$. A generalization to arbitrary linear forms was then obtained by Zaharescu in \mbox{\cite{zah}} (with a constant $c = \frac{1}{480}$), after which Meng and Zaharescu generalized it even further in \mbox{\cite{meza}}, to arbitrary linear forms in multiple variables. \\

Concerning the original problem, Erd\H os remarked in \mbox{\cite{erd}} that he was not able to find the optimal value of $c$. And as far as we are aware, in the better part of a century since, no improvements have occurred in the literature. In this paper we take another look at Erd\H os's proof, try to optimize its arguments, and find a better lower bound. \\

We start off by looking at upper bounds, however. We will prove that $f(n) \le \frac{n}{4} + O(1)$ holds for all $n \ge 4$, and conjecture that this is optimal.

\section{Upper bounds}
Recall that $\frac{a_1}{b_1}, \frac{a_2}{b_2}, \ldots$ is the Farey sequence of order $n$, and, in order to upper bound $f(n)$, we aim to find $k$ and $l > k$ with $(a_l - a_k)(b_l - b_k) < 0$ and $l-k$ as small as possible. We claim that such $k$ and $l$ exist with $l-k < \frac{n}{4} + 5$. 

\begin{upperbound} \label{upper}
For all $n \ge 4$ we have $f(n) \le \left \lfloor \frac{n}{4} \right \rfloor + d$, with $d = 1, 2, 2, 4$, depending on whether $n \equiv 0, 1, 2, 3 \pmod{4}$.
\end{upperbound}

To prove this theorem, we will use the following well-known property of consecutive Farey fractions.

\begin{standardfareylemma} \label{stdd}
Let $\frac{a}{b}$ and $\frac{c}{d}$ be two reduced fractions with $0 \le \frac{a}{b} < \frac{c}{d} \le 1$. Then they are consecutive fractions in the Farey sequence of order $n$ if, and only if, $bc - ad = 1$ and $\max(b, d) \le n < b + d$.
\end{standardfareylemma}

\begin{proof}[Proof of Theorem \ref{upper}]
If $n = 4m$ for some $m \in \mathbb{N}$, consider the fraction $\frac{a_k}{b_k} = \frac{2m-1}{4m}$. One can then check by Lemma \ref{stdd} that the Farey sequence continues as follows:
\begin{equation*}
\frac{m}{2m+1}, \frac{m+1}{2m+3}, \ldots, \frac{2m-1}{4m-1}, \frac{1}{2}, \frac{2m}{4m-1}.
\end{equation*}

With $\frac{a_l}{b_l}$ equal to this final fraction, we notice that $\frac{a_k}{b_k} = \frac{2m-1}{4m}$ and $\frac{a_l}{b_l} = \frac{2m}{4m-1}$ are not similarly ordered. Since $l = k + m + 2$, this shows $f(n) \le m + 1$. \\ 

If $n = 4m+1$ or $n = 4m+2$, consider $\frac{a_k}{b_k} = \frac{2m}{4m+1}$ instead. These are then the next Farey fractions:
\begin{equation*}
\frac{1}{2}, \frac{2m+1}{4m+1}, \frac{2m}{4m-1}, \ldots, \frac{m+1}{2m+1}, \frac{2m+1}{4m}.
\end{equation*}

With $\frac{a_l}{b_l} = \frac{2m+1}{4m}$ we have $l = k + m + 3$ and $(a_l - a_k)(b_l - b_k) < 0$, so that $f(n) \le m + 2$. \\

Finally, for $n = 4m+3$ we also take $\frac{a_k}{b_k} = \frac{2m}{4m+1}$ and $\frac{a_l}{b_l} = \frac{2m+1}{4m}$. In this case however, the two fractions $\frac{2m+1}{4m+3}$ and $\frac{2m+2}{4m+3}$ are contained in the sequence we just mentioned as well (right before and right after $\frac{1}{2}$ respectively). We therefore have $l = k + m + 5$, implying $f(n) \le m + 4$.
\end{proof}

Based on computer calculations we tentatively believe Theorem \ref{upper} to be optimal for large enough $n$. 

\begin{maxkconjecture} \label{conj}
For all $n \ge 4$ we have $f(n) > \frac{n}{4}$. More precisely, for all $n \ge 92$ we have the equality $f(n) = \left \lfloor \frac{n}{4} \right \rfloor + d$, with $d$ as in Theorem \ref{upper}.
\end{maxkconjecture}

We have checked this conjecture for all $n \le 5000$ and have not been able to find any counterexamples. In fact, the only positive integers $n$ with $4 \le n < 92$ for which $f(n)$ is strictly smaller than the upper bound from Theorem \ref{upper} are $n = 7, 9, 11, 15, 19, 23, 25, 27, 31, 35, 39, 49, 51, 63, 91$. \\

It is possible to strengthen the above conjecture in the following way: given any integer $d$, it seems plausible that for large enough $n$ one can actually classify all pairs of Farey fractions $(\frac{a_k}{b_k}$, $\frac{a_l}{b_l})$ with $l - k = \left \lfloor \frac{n}{4} \right \rfloor + d$ that are not similarly ordered. In particular, for every $d$ there should be an $e$ such that for all $n$ there are at most $e$ such pairs of fractions, with $e = 0$ for $d \le 0$ in particular. We leave the exact formulation (and proof) of such a stronger conjecture to the interested reader.

\section{Lower bounds}
To improve upon the lower bound $f(n) > \frac{n}{400}$ that was proven in \cite{erd}, we will first show that, given any fraction with small denominator, there is a small interval around it that only contains similarly ordered Farey fractions. To give an idea of what such an interval looks like, let us consider the fraction $\frac{4}{5}$. These are then the Farey fractions of order $40$ around this fraction:
\begin{equation*}
\frac{15}{19}, \frac{19}{24}, \frac{23}{29}, \frac{27}{34}, \frac{31}{39}, \frac{4}{5}, \frac{29}{36}, \frac{25}{31}, \frac{21}{26}, \frac{17}{21}.
\end{equation*}

One can notice that, to the left of $\frac{4}{5}$, both the numerators and the denominators form an increasing arithmetic progression (with common difference $4$ and $5$ respectively), whereas to the right of $\frac{4}{5}$ the numerators and denominators form decreasing arithmetic progressions. Such a result turns out to be true in general, which we will apply in the proof of our next lemma.

\begin{newfareylemma} \label{dontworryaboutsmalldenominators}
Let $\frac{a_k}{b_k}$, $\frac{a}{b}$ and $\frac{a_l}{b_l}$ be fractions in the Farey sequence of order $n$ with $\frac{a_k}{b_k} \le \frac{a}{b} \le \frac{a_l}{b_l}$. Then $\frac{a_k}{b_k}$ and $\frac{a_l}{b_l}$ are similarly ordered if $l - k \le \frac{n+b+1}{2b}$.
\end{newfareylemma}

\begin{proof}
If $b = 1$ the result is trivial as it forces either $\frac{a_k}{b_k} = \frac{0}{1}$ or $\frac{a_l}{b_l} = \frac{1}{1}$ in which case $\frac{a_k}{b_k}$ and $\frac{a_l}{b_l}$ are certainly similarly ordered, so without loss of generality we may assume $b \ge 2$. Moreover, if $\frac{n+b+1}{2b} < 3$, then the lemma follows from Mayer's result in \cite{may1}, so we may further assume $n \ge 5b-1$. Now, in the Farey sequence of order $b$, let $\frac{p}{q}$ and $\frac{r}{s}$ be the two fractions immediately to the left and right of $\frac{a}{b}$ respectively, and note that both $q$ and $s$ are smaller than $b$. Then, analogously to what we saw earlier in the case $\frac{a}{b} = \frac{4}{5}$, it follows from Lemma \ref{stdd} that the segment of the Farey sequence of order $n$ around $\frac{a}{b}$ is as follows:
\begin{equation*}
\frac{p + ca}{q + cb}, \frac{p + (c+1)a}{q + (c+1)b}, \ldots, \frac{p + da}{q + db}, \frac{a}{b}, \frac{r + d'a}{s + d'b}, \frac{r + (d'-1)a}{s + (d'-1)b}, \ldots, \frac{r + c'a}{s + c'b}.
\end{equation*}

Here, $c = \left \lfloor \frac{n - 2q - b}{2b} \right \rfloor + 1$, $c' = \left \lfloor \frac{n - 2s - b}{2b} \right \rfloor + 1$, $d = \left \lfloor \frac{n-q}{b} \right \rfloor$, and $d' = \left \lfloor \frac{n-s}{b} \right \rfloor$. The values of $c$ and $c'$ ensure that any sum of two consecutive denominators is larger than $n$ (which is required by Lemma \ref{stdd}), while $d$ and $d'$ are the largest values for which all denominators are smaller than or equal to $n$. \\

In order to prove Lemma \ref{dontworryaboutsmalldenominators}, we now have three different cases to consider: either $\frac{a_k}{b_k} = \frac{a}{b}$, or $\frac{a_l}{b_l} = \frac{a}{b}$, or $\frac{a_k}{b_k} < \frac{a}{b} < \frac{a_l}{b_l}$. As for the first case, it is clear that $\frac{a_k}{b_k} = \frac{a}{b}$ and $\frac{a_l}{b_l}$ are similarly ordered if $\frac{a_l}{b_l}$ is one of the elements in the segment, as both $a_l > a$ and $b_l > b$. Moreover, if $\frac{a_l}{b_l}$ is the smallest Farey fraction larger than $\frac{r + c'a}{s + c'b}$, then we claim $b_l > 2b$. Indeed, applying Lemma \ref{stdd} and $n \ge 5b-1$,
\begin{align*}
b_l &\ge n + 1 - (s + c'b) \\
&\ge n + 1 - \left(s + \frac{n - 2s - b}{2} + b\right) \\
&= \frac{n-b+2}{2} \\
&> 2b. 
\end{align*}

By the inequalities $s + c'b < (c' + 1)b \le 2c'b$ and the fact that $\frac{r + c'a}{s + c'b}$ and $\frac{a_l}{b_l}$ are consecutive Farey fractions, we (once again by Lemma \ref{stdd}) then get
\begin{align*}
a_l &= \frac{1 + b_l(r + c'a)}{s + c'b} \\
&> \frac{2bc'a}{2c'b} \\
&= a.
\end{align*}

Since both $a_l > a$ and $b_l > b$, we deduce that, even when $\frac{a_l}{b_l} > \frac{a}{b}$ is the smallest Farey fraction outside of the segment, $\frac{a}{b}$ and $\frac{a_l}{b_l}$ are still similarly ordered. We therefore conclude that $\frac{a_k}{b_k}$ and $\frac{a_l}{b_l}$ are similarly ordered in this case if $l - k \le d' - c' + 2$ holds, so in particular whenever $l - k \le \min(d-c, d'-c') + 2$. \\

Analogously, if $\frac{a_l}{b_l} = \frac{a}{b}$, then $\frac{a_k}{b_k}$ and $\frac{a_l}{b_l}$ are similarly ordered as well, as long as $l - k \le \min(d-c, d'-c') + 2$.  \\

As for the third and final case, assume that $\frac{a_k}{b_k} = \frac{p + ea}{q + eb}$ and $\frac{a_l}{b_l} = \frac{r + e'a}{s + e'b}$ are two fractions contained in the segment, with $\frac{a_k}{b_k} < \frac{a}{b} < \frac{a_l}{b_l}$, $c \le e \le d$ and $c' \le e' \le d'$. We then aim to prove that they are similarly ordered too. Define $X := a_l - a_k = r + e'a - p - ea$ and $Y := b_l - b_k = s + e'b - q - eb$. We then get
\begin{align*}
bX - aY &= (br - as) + (aq - bp) \\
&= 1 + 1.
\end{align*}

Here, the second equality follows from the fact that $\frac{p}{q}, \frac{a}{b}$ and $\frac{r}{s}$ were consecutive fractions in the Farey sequence of order $b$. Since $bX - aY = 2$ with $a \ge 1$ and $b \ge 2$, this implies that $X$ and $Y$ cannot have opposite signs, which is what we wanted to show. So in this third case we conclude that $\frac{a_k}{b_k}$ and $\frac{a_l}{b_l}$ are similarly ordered whenever $l - k \le \min(d-c, d'-c') + 2$ as well. \\

It therefore remains to calculate this latter quantity. By applying the aforementioned values of $c, c', d, d'$ we obtain
\begin{align*}
\min(d-c, d'-c') &= \min\left(\left \lfloor \frac{n-q}{b} \right \rfloor - \left \lfloor \frac{n - 2q - b}{2b} \right \rfloor, \left \lfloor \frac{n-s}{b} \right \rfloor - \left \lfloor \frac{n - 2s - b}{2b} \right \rfloor \right) - 1 \\
&\ge \min\left(\frac{n-q}{b} - \frac{n - 2q - b - 1}{2b}, \frac{n-s}{b} - \frac{n - 2s - b - 1}{2b} \right) - 2 \\
&= \frac{n+b+1}{2b} - 2.
\end{align*}

We conclude that if $l - k \le \frac{n+b+1}{2b}$, then $l - k \le \min(d-c, d'-c') + 2$, which in all three cases was sufficient to deduce that $\frac{a_k}{b_k}$ and $\frac{a_l}{b_l}$ are similarly ordered.
\end{proof}

Note that, in light of the proof of Theorem \ref{upper}, Lemma \ref{dontworryaboutsmalldenominators} is essentially optimal for $b = 2$. Now, before we continue with the statement and proof of our main lower bound, we need two more preliminary lemmas, where we define $N$ to be the number of Farey fractions of order $n$.

\begin{summatorytotient} \label{summtot}
For all positive integers $n$ we have $N > \frac{n^2}{4}$.
\end{summatorytotient}

\begin{proof}[Proof (sketch).]
With a computer one can check the inequality for all $n < 56$, so assume $n \ge 56$. With $\varphi(n)$ Euler's totient function, we have $N = 1 + \sum_{i \le n} \varphi(i)$. By applying M\"obius inversion to the identity $n = \sum_{d | n} \varphi(d)$ and rewriting the sum $\sum_{i \le n} \varphi(i)$, we obtain $N = 1 + \frac{1}{2} \sum_{i \le n} \mu(i) \left \lfloor \frac{n}{i} \right \rfloor \left(\left \lfloor \frac{n}{i} \right \rfloor + 1 \right)$. Since $\left \lfloor \frac{n}{i} \right \rfloor \left(\left \lfloor \frac{n}{i} \right \rfloor + 1 \right) > \frac{n^2}{i^2} - \frac{n}{i}$, $\sum_{i \ge 1} \frac{\mu(i)}{i^2} = \frac{6}{\pi^2}$ and $\sum_{i \le n} \frac{1}{i} < \log(n) + 1$, with some algebra one can deduce $N > \frac{3n^2}{\pi^2} - \frac{n}{2} \left(\log(n) + 2 \right)$ for all $n \ge 1$. Since the latter is larger than $\frac{n^2}{4}$ for $n \ge 56$, this finishes the proof.
\end{proof}

We will furthermore make use of the following tight result that was obtained by Dress in \cite{dress}.

\begin{dress} \label{dress}
For $\alpha \in [0, 1]$, let $A_n(\alpha)$ be the number of Farey fractions of order $n$ in the interval $(0, \alpha)$. For all $\alpha \in [0, 1]$ and all $n \in \mathbb{N}$ we then have the bounds 

\begin{equation*}
N\left(\alpha - \frac{1}{n}\right) \le A_n(\alpha) \le N\left(\alpha + \frac{1}{n}\right).
\end{equation*}
\end{dress}

We are now ready to prove our main lower bound.

\begin{lowerbound} \label{lower}
If $\frac{a_k}{b_k}$ and $\frac{a_l}{b_l} > \frac{a_k}{b_k}$ are two fractions in the Farey sequence of order $n$ with $l - k \le \frac{n}{12} \left(1 - \frac{4}{n^{1/3}} \right)$, then $\frac{a_k}{b_k}$ and $\frac{a_l}{b_l}$ are similarly ordered.
\end{lowerbound}

\begin{proof}
Taking the contrapositive, let us assume that $\frac{a_k}{b_k}$ and $\frac{a_l}{b_l}$ are not similarly ordered. We then see $\frac{a_l}{b_l} \ge \frac{a_k + 1}{b_k - 1} > \frac{a_k + 1}{b_k} \ge \frac{a_k}{b_k} + \frac{1}{n}$, so write $\frac{a_l}{b_l} - \frac{a_k}{b_k} = \frac{x}{n}$ for some $x > 1$. We now aim to show $l - k > \frac{n}{12} \left(1 - \frac{4}{n^{1/3}} \right)$, and by Lemma \ref{dontworryaboutsmalldenominators} we may assume $b_i > 6$ for all $i$ with $k \le i \le l$. We may further assume $n \ge 4^3 = 64$, as otherwise our upper bound is negative and the statement is trivially true. \\

Let $S_1$ be the set of indices $i$ with $k \le i \le l-1$ and $\min(b_1, b_{i+1}) \le \frac{n}{6}$, and let $S_2$ be those $i$ with $\min(b_1, b_{i+1}) > \frac{n}{6}$. Furthermore, let $i_1, i_2, \ldots, i_t$ be the actual indices for which $b_{i_j} \le \frac{n}{6}$. With these definitions in mind, we can show that we may assume that at least one of $b_{i_1}, b_{i_t}$ is larger than $n^{1/3}$.

\begin{notbothsmall} \label{notbothsmall}
If $n \ge 64$, $t \ge 2$, and $\max(b_{i_1}, b_{i_t}) \le n^{1/3}$, then $l - k > \frac{n}{2}$.
\end{notbothsmall}

\begin{proof}
If $\max(b_{i_1}, b_{i_t}) \le n^{1/3}$, then $\frac{a_{l}}{b_{l}} - \frac{a_{k}}{b_{k}} \ge \frac{a_{i_t}}{b_{i_t}} - \frac{a_{i_1}}{b_{i_1}} \ge \frac{1}{b_{i_1} b_{i_t}} \ge \frac{1}{n^{2/3}}$. Applying Lemma \ref{dress} with $\alpha = \frac{a_{k}}{b_{k}}$ and $\alpha =  \frac{a_{k}}{b_{k}} + \frac{1}{n^{2/3}}$, and we obtain that there are at least $N\left(\frac{1}{n^{2/3}} - \frac{2}{n} \right) = \frac{N(n^{1/3} - 2)}{n}$ Farey fractions in between $\frac{a_{k}}{b_{k}}$ and $\frac{a_{l}}{b_{l}}$. Since $\frac{N(n^{1/3} - 2)}{n} > \frac{n(n^{1/3} - 2)}{4}$ by Lemma \ref{summtot} and the latter is at least $\frac{n}{2}$ for $n \ge 64$, the proof is finished.
\end{proof}

With the help of Lemma \ref{notbothsmall} we can bound the sum of the reciprocals of the $b_{i_j}$.

\begin{sumofsmallreciprocals} \label{sumofsmallrecs}
We have the upper bound $$\sum_{j=1}^{t} \frac{1}{b_{i_j}} < \frac{x}{6} + \frac{1}{n^{1/3}}.$$
\end{sumofsmallreciprocals}

\begin{proof}
If $t = 1$, then we are done by the assumption $b_{i_1} > 6$. If $t > 1$, then
\begin{align*}
\frac{x}{n} + \frac{6}{n^{4/3}} &\ge \frac{6}{n^{4/3}} + \frac{a_{i_t}}{b_{i_t}} - \frac{a_{i_1}}{b_{i_1}} \\
&= \frac{6}{n^{4/3}} + \sum_{j=1}^{t-1} \left(\frac{a_{i_{j+1}}}{b_{i_{j+1}}} - \frac{a_{i_{j}}}{b_{i_{j}}} \right) \\
&\ge \frac{6}{n^{4/3}} + \sum_{j=1}^{t-1} \frac{1}{b_{i_j} b_{i_{j+1}}} \\
&\ge \frac{6}{n} \left(\max\left(\sum_{j=1}^{t-1} \frac{1}{b_{i_j}}, \sum_{j=2}^{t} \frac{1}{b_{i_j}}\right) + \frac{1}{n^{1/3}} \right) \\
&> \frac{6}{n} \sum_{j=1}^t \frac{1}{b_{i_j}},
\end{align*}

where the final inequality uses Lemma \ref{notbothsmall}. Multiplying both sides by $\frac{n}{6}$ gives the desired result.
\end{proof}

In the spirit of Erd\H os \cite{erd}, we will now write $\frac{x}{n}$ as the sum of two sums.

\begin{align*}
\frac{x}{n} &= \frac{a_l}{b_l} - \frac{a_k}{b_k} \\
&= \sum_{i=k}^{l-1} \left(\frac{a_{i+1}}{b_{i+1}} - \frac{a_i}{b_i} \right) \\
&= \sum_{i=k}^{l-1} \frac{1}{b_i b_{i+1}} \\
&= \sum_{i \in S_1} \frac{1}{b_i b_{i+1}} + \sum_{i \in S_2} \frac{1}{b_i b_{i+1}} 
\end{align*}

Applying $b_i + b_{i+1} > n$ for all $i$, we see that for the second sum (where $\min(b_i, b_{i+1}) > \frac{n}{6}$) we have $b_i b_{i+1} > \frac{n}{6} \frac{5n}{6} = \frac{5n^2}{36}$. This gives
\begin{equation*}
\sum_{i \in S_2} \frac{1}{b_i b_{i+1}} < \frac{36(l-k)}{5n^2},
\end{equation*}

or
\begin{equation*}
l - k > \frac{5n^2}{36} \sum_{i \in S_2} \frac{1}{b_i b_{i+1}}.
\end{equation*}

As for the first sum we have $b_i b_{i+1} > \min(b_i, b_{i+1}) \frac{5n}{6}$, while every element in $S_1$ occurs at most twice as an $i$ with $\min(b_i, b_{i+1}) \le \frac{n}{6}$. By furthermore applying Lemma \ref{sumofsmallrecs} we then get
\begin{align*}
\sum_{i \in S_1} \frac{1}{b_i b_{i+1}} &< \frac{6}{5n} \sum_{i \in S_1} \frac{1}{\min(b_i, b_{i+1})} \\
&\le \frac{12}{5n} \sum_{j=1}^t \frac{1}{b_{i_j}} \\
&< \frac{12}{5n} \left(\frac{x}{6} + \frac{1}{n^{1/3}} \right) \\
&= \frac{2x}{5n} - \frac{12}{5n^{4/3}}.
\end{align*}

We can now finish our proof as follows:
\begin{align*}
l-k &> \frac{5n^2}{36} \sum_{i \in S_2} \frac{1}{b_i b_{i+1}} \\
&= \frac{5n^2}{36} \left(\frac{x}{n} - \sum_{i \in S_1} \frac{1}{b_i b_{i+1}} \right) \\
&> \frac{5n^2}{36} \left(\frac{x}{n} - \frac{2x}{5n} - \frac{12}{5n^{4/3}} \right) \\
&= \frac{nx}{12} - \frac{n^{2/3}}{3} \\
&> \frac{n}{12} \left(1 - \frac{4}{n^{1/3}} \right). \qedhere
\end{align*}
\end{proof}

\section{A few final remarks}
The proof of Theorem \ref{lower} more generally shows the following result on the local density of Farey fractions.

\begin{generalthm}
Let $\frac{a_k}{b_k}$ and $\frac{a_l}{b_l}$ be two Farey fractions of order $n$ with $\frac{a_l}{b_l} - \frac{a_k}{b_k} = \frac{x}{n}$ for some $x > 0$. Then either there exists a Farey fraction $\frac{a}{b}$ with $b < \frac{6}{x}$ and $\frac{a_k}{b_k} \le \frac{a}{b} \le \frac{a_l}{b_l}$, or $l - k > nx \left(\frac{1}{12} - o(1) \right)$.
\end{generalthm}

However, one can check that a direct application of Lemma \ref{dress} already improves upon this more general theorem for $x > 2.76$, so its value seems to stem mostly from small values of $x$. \\

And on that note, for $\frac{a_k}{b_k} \ge \frac{1}{2} - o(1)$ we have $x \ge \frac{3}{2} - o(1)$ if $\frac{a_k}{b_k}$ and $\frac{a_l}{b_l}$ are not similarly ordered. In this case we get the improved lower bound $l - k > n \left(\frac{1}{8} - o(1)\right)$ which in turn is at most a factor $2$ off from optimal, by the proof of Theorem \ref{upper}.


\begin{thebibliography}{9}

\bibitem{may1}
	A. E. Mayer,
	\emph{A mean value theorem concerning Farey series}.
	The Quarterly Journal of Mathematics, Volume os-13, Issue 1, 48--57,
	1942.
	
\bibitem{may2}
	A. E. Mayer,
	\emph{On neighbours of higher degree in Farey series}.
	The Quarterly Journal of Mathematics, Volume os-13, Issue 1, 185--192,
	1942.

\bibitem{erd}
	P. Erd\H os,
	\emph{A note on Farey series}.
	The Quarterly Journal of Mathematics, Volume os-14, Issue 1, 82--85,
	1943.
	Also available \href{https://users.renyi.hu/~p_erdos/1943-01.pdf}{here}.

\bibitem{zah}
	A. Zaharescu,
	\emph{The Mayer-Erd\H os phenomenon}.
	Indagationes Mathematicae, Volume 17, Issue 1, 147--156,
	2006.
	Also available \href{https://www.sciencedirect.com/science/article/pii/S0019357706800121}{here}.
	
\bibitem{meza}
	X. Meng, A. Zaharescu,
	\emph{A multivariable Mayer-Erd\H os phenomenon}.
	Journal of the Korean Mathematical Society, Volume 51, Issue 5, 1029--1044,
	2014.
	Also available \href{https://koreascience.kr/article/JAKO201426636276924.pdf}{here}.

\bibitem{dress}
	F. Dress,
	\emph{Discrépance des suites de Farey}.
	Journal de th\'eorie des nombres de Bordeaux, Volume 11, Issue 2, 345--367,
	1999.
	Also available \href{https://www.numdam.org/article/JTNB_1999__11_2_345_0.pdf}{here}.
	
\end{thebibliography}
\end{document}